\documentclass{elsarticle}
\usepackage{epsfig}
\usepackage{epic}
\usepackage{eepic}
\usepackage{enumerate}
\usepackage{graphics}
\usepackage{floatrow}
\usepackage{subfigure}
\usepackage[english]{babel}
\usepackage[utf8]{inputenc}
\usepackage{amsmath}
\usepackage{amssymb}
\usepackage{amsthm}
\usepackage{enumerate}
\usepackage{graphicx}
\usepackage[colorinlistoftodos]{todonotes}

\newtheorem{thm}{Theorem}[section]
\newtheorem{lemma}[thm]{Lemma}

\newtheorem{claim}[thm]{Claim}
\newtheorem{obs}[thm]{Observation}

\newcommand{\K}{\mathcal{K}} 
\newcommand{\MD}{\mathcal{MD}} 
\newcommand{\SD}{\mathcal{SD}}
\newcommand{\st}{\colon\ }
\newcommand{\B}{\mathcal{B}}
\newcommand{\DG}{\mathfrak{D}_{G}}
\newcommand{\dd}{\mathfrak{d}}
\newcommand{\A}{\mathcal{A}}

\newcommand{\sm}{\setminus}

\begin{document}


\begin{frontmatter}

\title{Cops and robbers on oriented toroidal grids}

\author[SFU]{Sebasti\'an Gonz\'alez Hermosillo de la Maza
}
\ead{sga89@sfu.ca}

\author[SFU]{Seyyed Aliasghar Hosseini}
\ead{sahosseini@sfu.ca}

\author[SFU]{Fiachra Knox\fnref{knox}}
\ead{fknox@sfu.ca}

\author[SFU]{Bojan Mohar\fnref{mohar,mohar2}}
\ead{mohar@sfu.ca}

\author[McGill]{Bruce Reed\fnref{reed}}
\ead{breed@cs.mcgill.ca}

\address[SFU]{Simon Fraser University\\
						8888 University Drive\\
						Burnaby, BC, Canada\\}
						
\address[McGill]{School of Computer Science \\
			McGill University \\
			3480 University \\
			Montreal, Quebec, Canada H3A 2A7 }		

 \fntext[knox]{Supported by a PIMS post-doctoral fellowship.}
 \fntext[mohar]{Supported in part by an NSERC Discovery Grant R611450 (Canada),
   by the Canada Research Chair program, and by the
    Research Grant J1-8130 of ARRS (Slovenia).}
\fntext[mohar2]{On leave from:
    IMFM \& FMF, Department of Mathematics, University of Ljubljana, Ljubljana,
    Slovenia.}
 \fntext[reed]{Supported by an NSERC Discovery Grant.}
 
\begin{abstract}
The game of cops and robbers is a well-known game played on graphs. In this paper we consider the straight-ahead orientations of $4$-regular quadrangulations of the torus and the Klein bottle and we prove that their cop number is bounded by a constant. We also show that the cop number of every $k$-regularly oriented toroidal grid is at most 13.
\end{abstract}


\end{frontmatter}
 

\section{Introduction}

\subsection{Basic definitions}

The game of Cops and Robbers is a well-known game on graphs that was  introduced by Nowakowski and Winkler \cite{nowakowski} and Quilliot \cite{quilliot}.
There are two players, one controls the cops and the other one controls the robber. The game starts with the cops selecting some vertices as their initial positions (multiple cops can select the same vertex). Then the robber chooses his initial position. From now on, first the cops move and then the robber moves, where moving means either staying at the same position or moving to a neighbouring vertex. The game on digraphs is defined in the same way except that every move must be made along a directed edge.

The cops win the game if one of them can get to the same vertex as the robber and the robber wins if he can avoid this indefinitely. The minimum number of cops that can guarantee the robber's capture in a graph or digraph $G$ is called the \emph{cop number} of $G$ and will be denoted by $c(G)$.

Game of cops and robbers on graphs has received lots of attention but very little is known about the game on digraphs. Basic results have been introduced in \cite{Bonato}. The game of Cops and Robbers in digraphs can be as natural and as inspiring as the game on undirected graphs.  

Recall that a (di)graph $\widehat Q$ is a \emph{cover} over a (di)graph $Q$ if there is a (di)graph homomorphism $\pi: \widehat Q\to Q$ (called the \emph{covering projection}) which maps the edges incident with any vertex $v$ bijectively onto the edges incident with $\pi(v)$ in $Q$. 
The following lemma from \cite{grid}  can be used to analyze the cops and robber game on $Q$ where a good strategy of cops on $\widehat Q$ is known.

\begin{lemma}\cite{grid}
\label{lem:cover}
Suppose that a graph or digraph $\widehat Q$ is a cover over a (di)graph $Q$. If $k$ cops have winning strategy on $\widehat Q$, then they also win on $Q$, that is $c(Q)\le c(\widehat Q)$.
\end{lemma}

Throughout the paper we will stick with standard graph theory terminology and notation. By $P_n$ and $C_n$ we define the path and the cycle of order $n$, respectively.  The $k\times \ell$ \emph{grid} is the Cartesian product of paths $P_k\Box P_\ell$ and the $k\times \ell$ \emph{toroidal grid} is the Cartesian product of cycles $C_k\Box C_\ell$. Every toroidal grid has a natural embedding into the torus such that each face is bounded by a 4-cycle of the grid.

\subsection{4-regular quadrangulations}

A graph together with a 2-cell embedding in a surface with all faces bounded by 4-cycles is called a \emph{quadrangulation} of that surface. Our original motivation is in discussing the game of cops and robbers on 4-regular quadrangulations but, as shown in \cite{grid} and overviewed below, the discussion can be restricted to the case of toroidal grids via surface coverings. We do not intend to go into details, and refer to \cite{Stillwell} concerning the notions about covering spaces of surfaces used only in this short subsection.

Consider a 4-regular quadrangulation of a surface. It follows by Euler's formula that the surface is either the torus or the Klein bottle and it can be shown by using the Gauss-Bonnet Theorem that the straight-ahead walks partition the set of edges into cycles, all of which are noncontractible on the surface. These cycles can be split into two classes, each class consisting of pairwise disjoint cycles (we call them \emph{vertical cycles} and \emph{horizontal cycles}, respectively) such that each vertical and each horizontal cycle intersect (possibly more than once). By giving each of these cycles an orientation, we obtain an Eulerian digraph in which, at each vertex, the two incoming edges and two outgoing edges are consecutive in the local rotation around the vertex.
The universal cover of a 4-regular quadrangulation is the 4-regular tessellation of the plane with square faces (the integer grid), and every finite quotient of the integer grid is a 4-regular quadrangulation of the torus or the Klein bottle.
Four-regular quadrangulations of the torus admit a simple description. Each such quadrangulation is of the form $Q(r,s,t)$, where $r,s,t$ are arbitrary positive integers, $0\le t < r$, and $Q(r,s,t)$ is obtained from the $(r+1)\times (s+1)$ grid with underlying graph $P_{r+1}\square P_{s+1}$ by identifying the ``leftmost" path of length $s$ with the ``rightmost" one (to obtain a cylinder) and identifying the bottom $r$-cycle of this cylinder with the top one after rotating the top clockwise for $t$ edges.  This classification can be derived by considering appropriate fundamental polygon of the universal cover (which is isomorphic to the tessellation of the plane with squares).  Quadrangulations of the Klein bottle are a bit more complicated. While all toroidal quadrangulations $Q(r,s,t)$ are vertex-transitive maps, this is no longer true for the Klein bottle. For our purpose it will suffice to know that the orientable double cover $\widetilde Q$ of such a quadrangulation $Q$ is a quadrangulation of the torus and is therefore isomorphic to some $Q(r,s,t)$. This implies that any cover $C_n\Box C_n$ of $Q(r,s,t)$ is also a cover of $\widetilde Q$.

\subsection{Our main results}

The main result of this paper is a constant upper bound on the cop number of 4-regular quadrangulations with an arbitrary straight-ahead orientation. 

\begin{thm} \label{thm:4-regular quadrangulations 319} 
If $G$ is any 4-regular quadrangulation of a surface endowed with a straight-ahead orientation of its edges, then $c(G) \leq 319$.
\end{thm}

The above-mentioned results imply that every 4-regular quadrangulation of the torus or the Klein bottle has toroidal grid $C_n \square C_n$ as its covering graph.
By Lemma \ref{lem:cover}, we will be able to restrict ourselves to deal with different orientations of the toroidal grid $C_n \square C_n$, which is given later in the paper as Theorem \ref{thm:any directed cycles}. That result proves Theorem \ref{thm:4-regular quadrangulations 319}.

The paper \cite{grid} treats the cop number of orientations of $C_n \square C_n$ for which the corresponding digraph is vertex-transitive. In this paper we consider more general straight-ahead orientations, and show that their cop number is bounded by 319. While this bound is not small, we also discuss another case in which the the orientations are a bit more restricted, and we obtain a much better bound. See Theorem \ref{k-regular}.

\section{Streams, confluxes, and traps}

The grid we are working on is $C_n \square C_n$. For a vertex $v = (x,y) \in C_n \square C_n$, the digraph induced by $\{(x,z)\st z \in C_n\}$ will be called the \emph{row} of $v$, and the subdigraph induced by $\{(z,y)\st z \in C_n\}$ will be called the \emph{column} of $v$. A \emph{line} containing a vertex $v$ is either a row or a column of $v$. 
We will assume that the edges of $C_n \square C_n$ are oriented in such a way that every line is a directed $n$-cycle in the grid. We refer to this as a \emph{straight-ahead orientation} of the toroidal grid.
We say that two lines in $C_n \square C_n$ are \emph{consecutive} if one of their coordinates correspond to adjacent vertices in one of the factors of $C_n \square C_n$. A set $S$ of consecutive lines oriented in the same direction will be called a \emph{stream}\index{stream} and its width $w(S)$ is the number of lines in the stream. If $S$ and $S'$ are streams such that $S' \subseteq S$, we say that $S'$ is a substream of $S$. 

If $S_1$ and $S_2$ are edge-disjoint streams and the set $\K = V(S_1) \cap V(S_2)$ is not empty, we will call $\K$ a \emph{conflux}\index{conflux} (see Figure \ref{conflux}). The vertices in a conflux $\K$ with the minimum number of neighbours in $\K$ are called \emph{corners}\index{corner}.
Notice that the set of corners of a conflux $\K = S_1 \cap S_2$ is never empty, and if $V(\K) \neq V(C_n \square C_n)$, then it can have  four vertices (if $w(S_1),w(S_2) \geq 2$), two vertices ($w(S_i)=1$ and $w(S_j) \geq 2$ with $\{i,j\} = \{1,2\}$) or one vertex ($w(S_1) = w(S_2) = 1$). 

If $\K$ has four corners, then a corner is \emph{main} if it has an odd number of outneighbours in $\K$ and \emph{secondary} otherwise (see Figure \ref{conflux}). However, if $\K$ has one or two corners, they will all be referred to as main.

\begin{figure}[ht!]
\begin{center}
\includegraphics{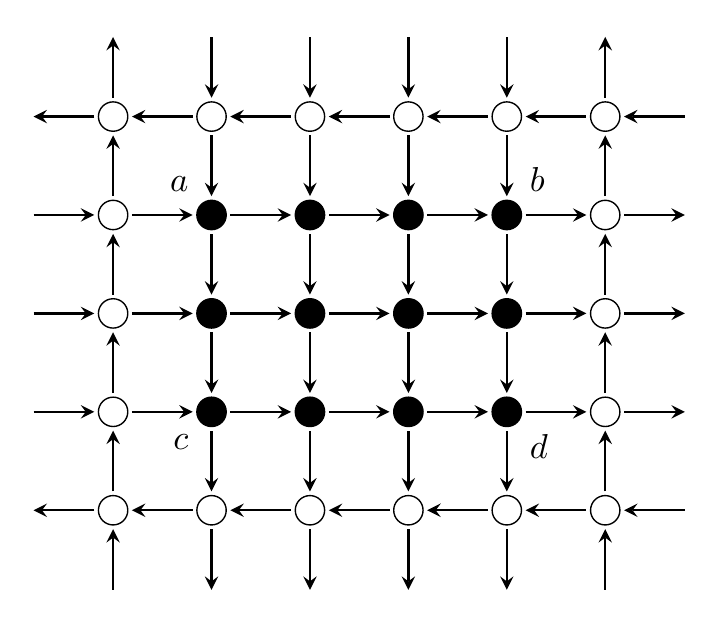}
\caption{The black vertices form a (maximal) conflux. Corners $b$ and $c$ of the conflux are its main corners, while $a$ and $d$ are secondary corners.} \label{conflux}
\end{center}
\end{figure}

We will always assume that the robber is forced to move from its current position. We can make sure this happens by chasing him with a cop. For Lemma \ref{trap1} and Lemma \ref{trap2}, we will assume that one cop is chasing the robber so the robber is forced to move. We will use $p(R)$ to denote the current position of the robber, and $p(C_i)$ for the position of the cop $C_i$.

\begin{lemma}\label{trap1}

Let $\K$ be a conflux with one cop on each main corner. If $p(R) \in V(\K)$ and $N^-(p(R)) \nsubseteq V(\K)$, then the robber will be captured or his movements will be restricted to a stream containing $\K$.

\end{lemma}

\begin{proof}
Without loss of generality we may assume the incoming edge to $p(R)$ comes from above due to the symmetry of the argument. 
Notice that when $N^-(p(R)) \nsubseteq V(\K)$, there is a main corner cop in the same row as $p(R)$. Let us call this cop $C_1$ and the other one $C_2$. In order to leave $\K$, the robber must step on the column where $C_1$ is or leave through the bottom. 

The strategy for $C_1$ and $C_2$ will be the following: If the robber moves towards $C_1$'s column, then $C_1$ stays where he is and $C_2$ copies the robber's move and moves towards $C_1$'s column. 

If the robber moves down, then $C_1$ moves down with the robber, and $C_2$
\begin{itemize}
    \item stays in the same place if it is in the same column but different row as the robber,
    \item moves down with the robber if he is in the same row but different column,
    \item and moves towards the robber's column otherwise.
\end{itemize}

By following this strategy, the robber and $C_1$ are always on the same row, so the robber cannot leave $\K$ crossing $C_1$'s column or he will be captured. Therefore the only why for the robber to exit $\K$ is to move down and leave $\K$ crossing $C_2$'s row. Note that when the robber is in the same row as $C_2$, then by the above strategy, $C_1$ is also in the same row. From now on, $C_1$ and $C_2$ will stay in the same row as the robber. This means that the robber's movements are restricted to $S$, the stream formed by the columns containing vertices of $\K$.
\end{proof}

Notice that in the case where the streams that form $\K$ have the same width, two cops guarantee the capture of the robber (or force the robber to stay still). However, once the robber's movements have been restricted to a stream, one extra cop will guarantee the capture. Note that for Lemma \ref{trap1} to work, we need to set up the trap before the robber enters it. It is possible to set a slightly different trap that works regardless of where the robber's in-neighbors are, but we need one more cop to do this.

\begin{lemma}\label{trap2}

Let $\K$ be a conflux with one cop on each main corner and one cop in the secondary corner of $\K$ without out-neighbors in $\K$ if such corner exists. If the robber is in $\K$, then he will be captured or his movements will be restricted to a stream containing $\K$.

\end{lemma}

\begin{proof}

Let $S_1$ and $S_2$ be the streams such that $S_1 \cap S_2 = \K$. If we have that $\min \{w(S_1), w(S_2)\} = 1$, then all the corners of $\K$ are main and by Lemma \ref{trap1} we are done. 
If $\min \{w(S_1), w(S_2)\} > 1$ and $v$ is the secondary corner of $\K$ with a cop, take the substreams $S_1'$ and $S_2'$ of $S_1$ and $S_2$, respectively, that contain just the row (column) through the vertex $v$. After at most $w(S_1)+w(S_2)$ moves, the robber will be on a vertex of $\K_1 = V(S_1) \cap V(S_2')$ or $\K_2 = V(S_1') \cap V(S_2)$. Since the main corners of both $\K_1$ and $\K_2$ are covered by cops, then an application of Lemma \ref{trap1} gives us the desired result.
\end{proof}


Given a grid $G$ (with vertex set $V$ and arc set $A$), we can define the conflux digraph of $G$, which we will denote by $\mathfrak{D}_{G}$, as the digraph whose vertex set $V(\mathfrak{D}_{G})$ consists of all maximal confluxes of $G$, and where $(\K_1,\K_2)$ is an edge of $\mathfrak{D}_{G}$ whenever there exist vertices $u \in \K_1$ and $v \in \K_2$ such that $(u,v) \in A$. Note that $\mathfrak{D}_{G}$ is isomorphic to $C_k \square C_l$ (with straight-ahead orientation) where $k$ is the number of maximal column streams and $l$ is the number of maximal row streams.

There is a natural correspondence between $C_n \square C_n$ and the elements of $\mathbb{Z}_n \times \mathbb{Z}_n$. Select an arbitrary vertex $v\in V(C_n\Box C_n)$ and associate it with the element $(0,0)$ in $\mathbb{Z}_n \times \mathbb{Z}_n$. If $(x,y)$ corresponds to a vertex, set its right neighbour to be $(x+1$ mod $n,y)$ and its top neighbour to be $(x,y+1$ mod $n)$. This correspondence allows us to represent each move of the robber or a cop by the addition of a vector in $\{(0,0),(1,0), (0,1), (-1,0), (0,-1)\}$ to its current position.

Given a vertex $v \in \mathbb{Z}_n \times \mathbb{Z}_n$ with $N^+(v) = \{u,w\}$, we can define the sets $$\SD(v) = \{x \in \mathbb{Z}_n \times \mathbb{Z}_n\st x - v = r(u + w - 2v), \textnormal{ for some } r \in \mathbb{Z}\},$$ $$\MD(v) = \{x \in \mathbb{Z}_n \times \mathbb{Z}_n\st x - v = r(u - w), \textnormal{ for some } r \in \mathbb{Z}\}.$$ 
\begin{obs}\label{transversal}
For any vertex $v$ and any line $L$ in $G$, $\SD(v) \cap L \neq \varnothing$ and $\MD(v) \cap L \neq \varnothing$.
\end{obs}

The sets $\SD(v)$ and $\MD(v)$ will be called the \emph{secondary diagonal} and \emph{main diagonal} of $v$, respectively.
Geometrically speaking, if we think of the arcs of the digraph as vectors, $\SD(v)$ is the set of all the vertices of $G$ in the diagonal line through $v$ defined by the sum of the arcs leaving $v$, and $\MD(v)$ is the set of vertices in the line orthogonal to that one. Notice that for every vertex $v \in V$ with $N^+(v) = \{u,w\}$, we have that $u + w - 2v$ is an element of $\{1,-1\}^2$. 
This value will be called the \emph{type} of the vertex, $\tau (v)$, and two vertices will be of opposite types if their types are additive inverses in $\mathbb{Z}^2$. Elements of $\{1,-1\}^2$ will be refered to as \emph{types}. Notice that all the vertices in a conflux $\K$ have the same type, so we can define $\tau(\K) = \tau(v)$ where $v \in \K$.

As an example let us consider Figure \ref{conflux} and assume that $c$ corresponds to $(x,y)$. Now $\tau(c) = (x+1,y) + (x,y-1) - 2(x,y) = (1,-1)$. This is the type of all black vertices in Figure \ref{conflux}.
%
%

Let $v$ be a vertex in $G$ and $\K$ the maximal conflux containing $v$. We define the \emph{horizontal escape distance} of $v$, $\mathcal{HE}(v)$ as the length of the shortest directed path starting at $v$ and ending at a vertex outside of $\K$ using only horizontal arcs (adding $(\pm 1,0)$). Analogously, we define the \emph{vertical escape distance} of $v$ and denote it with $\mathcal{VE}(v)$. The escape distance of $v$ is $\mathcal{E}(v) = \min\{\mathcal{HE}(v),\mathcal{VE}(v)\}$.

\begin{lemma}

Let $\K_1, \K_2, \K_3$ and $\K_4$ be confluxes of $G$ such that $N^+(\K_1) \cup N^+(\K_3) \subseteq \K_2 \cup \K_4$. If there are cops in the main corners of $\K_2$ and $\K_4$, and the robber is in $\K_1 \cup \K_3$, then the robber will be captured or its movements will be restricted to a stream.
\end{lemma}

\begin{proof}
It is easy to see that if the robber is in $\K_1 \cup \K_3$ and is forced to move, then he will enter $\K_2 \cup \K_4$. Since the main corners of $\K_2$ and $\K_4$ are covered, the result follows from Lemma \ref{trap1}.
\end{proof}

\section{The $k$-regularly oriented grid}

We say that a grid $G = C_n \square C_n$ is \emph{$k$-regularly oriented} if $w(S) = k$ for every maximal stream $S$ in $G$. The cases where $k \in \{1, n\}$ have been covered in \cite{grid}, so in this section we will assume that $G$ is a $k$-regularly oriented grid with $2\leq k < n$. Let $v$ and $w$ be vertices in $G$. We say $w$ is a \emph{main shadow}\index{shadow} of $v$ if:

\begin{enumerate}

\item[\it{i)}] $w \in \MD(v)$,

\item[\it{ii)}] $\tau(v) = \tau(w)$, and

\item[\it{iii)}] $\mathcal{VE}(v) =  \mathcal{HE}(w)$.

\end{enumerate}
We say that $w$ is a \emph{secondary shadow} of $v$ if:

\begin{enumerate}

\item[\it{i)}] $w \in \SD(v)$,

\item[\it{ii)}] $\tau(v) = -\tau(w)$, and

\item[\it{iii)}] $\mathcal{VE}(v) =  \mathcal{HE}(w)$.

\end{enumerate}
Notice that we get equivalent definitions by changing condition {\it iii)} for 

\begin{itemize}

\item[\it{iii')}]$\mathcal{HE}(v) =  \mathcal{VE}(w)$.

\end{itemize}

In the case when $p(R) = v$, we will call $w$ a \emph{main} (or \emph{secondary}) \emph{shadow of the robber}.
We say a vertex $w$ is a \emph{diagonal shadow} of a vertex $v$ if $w$ is a secondary shadow of $v$ or a main shadow of $v$. Again, if $p(R) = v$ we will use the term \emph{diagonal shadow of the robber}. The following result states that if a cop is in a diagonal shadow of the robber and the robber moves, there is a move that the cop can make that keeps him in a diagonal shadow of the robber. Notice that if the type of the vertex the robber is in changes when he moves, the diagonal that the cop must be in will change from secondary to main, or vice versa.

\begin{lemma}\label{closeshadow}
Let $v, u, w \in V(G)$ be vertices such that $N^+(v) = \{u,w\}$ and take $d = w - v$. 
\begin{itemize}
\item If $x$ is a main shadow of $v$, then $y = x + d$ is a diagonal shadow of $u$ and $y \in N^+(x)$.
\item If $x$ is a secondary shadow of $v$, then $y = x - d$ is a diagonal shadow of $u$ and $y \in N^+(x)$.
\end{itemize} 
\end{lemma}

The proof is left to the reader. Figure \ref{fig:new} is added for the guidance.

\begin{figure}[htb]
\begin{center}
\includegraphics[totalheight=0.28\textheight]{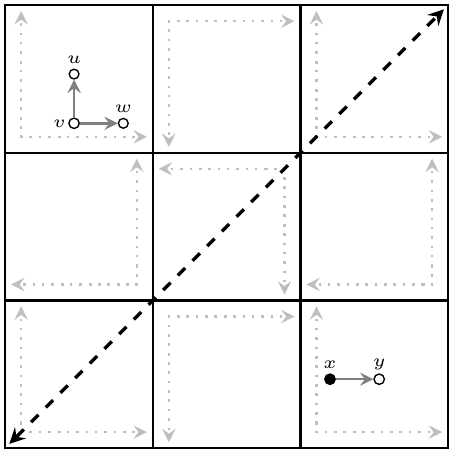}\hspace{1cm}
\includegraphics[totalheight=0.28\textheight]{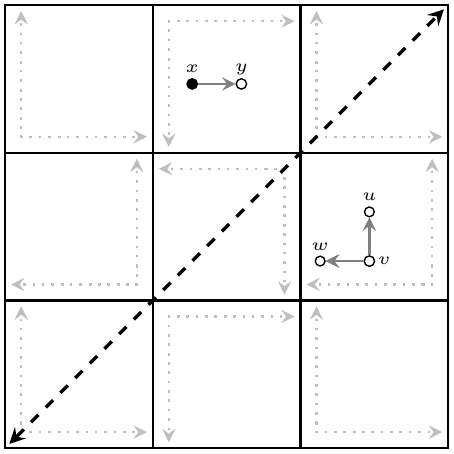}
\caption{Proof of Lemma \ref{closeshadow}. Note that in the left figure, $x$ is a main shadow of $v$, and $y$ is a main shadow of $u$ (they have the same type). In the right figure, $x$ is a secondary shadow of $v$, and $y$ is a secondary shadow of $u$ (they have opposite types).} \label{fig:new}
\end{center}
\end{figure}

The next lemma will give us another way of restricting the robber's moves in the case of a $k$-regularly oriented grid. 

\begin{lemma}\label{bisector}

Let $u,v,s,t,x,y$ be vertices such that $u \neq v$, $v$ is a diagonal shadow of $u$, $s \in N^+(u)$, $t \in N^+(v)$ and is a diagonal shadow of $s$, $x$ is the intersection of the row of $u$ and the column of $v$, and $y$ is the intersection of the row of $s$ and the column of $t$. If $d \in \{1,-1\}^2$ is orthogonal to $v-u$, then there exists an integer $r$ such that $x = y +rd$.

\end{lemma}
 
\begin{proof}
Let $A = \{(1,0), (-1,0)\}$ and $B = \{(0,1), (0,-1)\}$. If $u - s \in A$, then $v - t \in B$, and therefore $x = y$ and $r = 0$. If $u - s \in B$, then $v - t \in A$. In this case, we can see that $y = x + (u - s) + (v - t)$. Since $(u - s) + (v - t) \in \{1,-1\}^2$, we get that there is an integer $r\in \{-1,0,1\}$ such that $x = y +rd$.
\end{proof}
 
\begin{figure}[ht!]
\begin{center}
\includegraphics[totalheight=0.28\textheight]{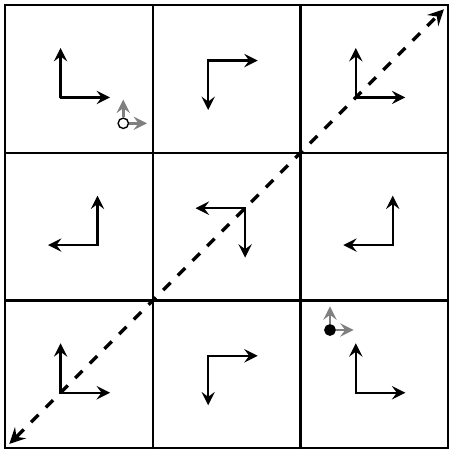}\hspace{1cm}
\includegraphics[totalheight=0.277\textheight]{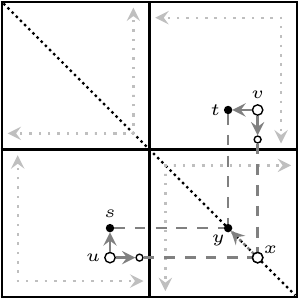}
\vspace{1mm}

{\small (a) \hspace{5.9cm} (b)}
\caption{(a) The robber (white vertex $u$) is not able to cross the mirror line (dashed line) or he will be caught by the cop (black vertex $v$). (b) One of the possible cases for Lemma \ref{bisector}.} \label{diagonal}
\end{center}
\end{figure}


The geometric interpretation of Lemma \ref{bisector} is key (see Figure \ref{diagonal}b): If a cop moves in such a way that he remains in a diagonal shadow of the robber, then there exists a set of vertices that touches every line in $G$ (Observation \ref{transversal}) that the robber cannot step on (and therefore cannot cross it) without being captured. 
This set corresponds to the vertices in the orthogonal bisector of the ``line segment'' from $u$ to $v$ and is shown in the figure as the black dotted diagonal line through $x$ and $y$. We will refer to this line as the \emph{mirror} of the corresponding shadow of the robber. Two mirrors $\ell$ and $\ell'$ are parallel if the two types of their vertices are the same. 

Recall that we are working on a $k$-regularly oriented grid. Given two parallel mirror lines $\ell$ and $\ell'$, and $d$ a type orthogonal to the type of a vertex in $\ell$, the distance between $\ell$ and $\ell'$ , denoted by $\A(\ell,\ell')$, is the minimum positive integer $m$ such that $\ell = \ell' + mkd$ or $\ell = \ell' - mkd$. Notice that $\A(\ell,\ell') \geq 2$ for any two different parallel mirrors, $\ell$ and $\ell'$.

Let $v$ be a vertex of $G$, $s \in \mathbb{N}$ and $d$ a type orthogonal to $\tau(v)$. We define $$\MD_s(v,d) = \{x \in V(G) \st x - \frac{s(\tau(v)+d)}{2} \in \MD(v)\}.$$ Because of Observation \ref{transversal}, we know that every vertex $u \in G$ is in $\MD_s(v,d)$ for some $s \in \mathbb{N}$ and some type $d$. 
Given two vertices of opposite types, $u$ and $v$, the \emph{diagonal distance}, which we will denote by $\dd(u,v)$, between $u$ and $v$ is the minimum $t \in \mathbb{N}$ such that $v \in \MD_t(u,d) \cup \MD_t(u,-d)$ (or equivalently, $u \in \MD_t(v,d) \cup \MD_t(v,-d)$). 
Notice that $v \in \MD_t(u,d)$ if and only if $u \in \MD_t(v,-d)$. With $t=\dd(u,v)$, we define 
$$
  \B(u,v) = \bigcup_{i=0}^{t} \MD_i(u,d),
$$ 
where $d$ has been chosen so that $v \in \MD_{t}(u,d)$. Notice that if $\K$ and $\K'$ are confluxes of $G$ such that $\tau(\K) = -\tau(\K')$, then $\B(\K,\K')$ and $\dd(\K, \K')$ are defined in $\DG$.

\begin{obs}
\label{evendd}
If $\K$ and $\K'$ are maximal confluxes of a $k$-regularly oriented grid $G$, and $\tau(\K) = -\tau(\K')$, then we have $\dd(\K,\K') < \frac{n}{k}$ and $\dd(\K,\K') \in 2\mathbb{Z}$.
\end{obs}

For the rest of the section, all the confluxes will be assumed to be maximal. For the lemmas \ref{land0}, \ref{land1}, \ref{land2} and \ref{land3},  $\K_1$ and $\K_2$ will denote  confluxes, whose main corners are covered by cops.

\begin{lemma}\label{land0}

Suppose $\tau(\K_1) = -\tau(\K_2)$ and that the robber is in a conflux in $V(\DG) - \B(\K_1,\K_2)$ whose type is orthogonal to $\tau(\K_1)$. If the robber enters $\B(\K_1,\K_2)$, then the cops can capture a main shadow of the robber.

\end{lemma}

\begin{proof}

If the robber is in $V(\DG) - \B(\K_1,\K_2)$, in order to enter $\B(\K_1,\K_2)$ the robber must enter a conflux $\K \in \MD(\K_1) \cup \MD(\K_2)$. In either case, a main shadow of the robber will enter $\K_1 \cup \K_2$. Since the main corners of both $\K_1$ and $\K_2$ are covered by cops, we capture a main shadow of the robber by applying Lemma \ref{trap1} to the corresponding shadow.
\end{proof}

\begin{lemma}\label{land1}

Let $d$ be a type orthogonal to $\tau(\K_1)$, and suppose that $\K_2$ satisfies $\tau(\K_1) = -\tau(\K_2)$ and $\dd(\K_1,\K_2) = 2$. If the robber is in $\B(\K_1,\K_2)$, then we can force him to move (using an extra cop, a total of 5) to a conflux outside of $\B(\K_1,\K_2)$.

\end{lemma}

\begin{proof}

If the robber is at a vertex of type $d$ or $-d$ and is in $\B(\K_1,\K_2)$, then by forcing him to move he must enter a conflux $\K$ whose type is parallel to $\tau(\K_1)$.
If $\tau(\K) = \tau(\K_1)$, then $\K \in \MD(\K_2)$, and so by forcing him to move he will exit $\B(\K_1,\K_2)$. If $\tau(\K) = -\tau(\K_1)$, then we have $\K \in \MD(\K_1)$, so he will exit $\B(\K_1,\K_2)$ when we force him to move.
\end{proof}

\begin{lemma}\label{land2}

Let $\K_1, \K_2$ and $d$ be the same as in the hypothesis of Lemma \ref{land1} and take $\K_3$ such that $\tau(\K_3) = \tau(\K_1)$ and $\K_3 = \K_2 +\tau(\K_1) + r'd$ in $\mathfrak{D}_{G}$ for some $r' \in \mathbb{Z}$.
If the robber is in $\B(\K_1,\K_3)\setminus\B(\K_1,\K_2)$, then we can force him to move to a vertex out of $\B(\K_1,\K_3)$ or we capture his main shadow.
\end{lemma}

\begin{figure}[htb]
\begin{center}
\includegraphics[totalheight=0.4\textheight]{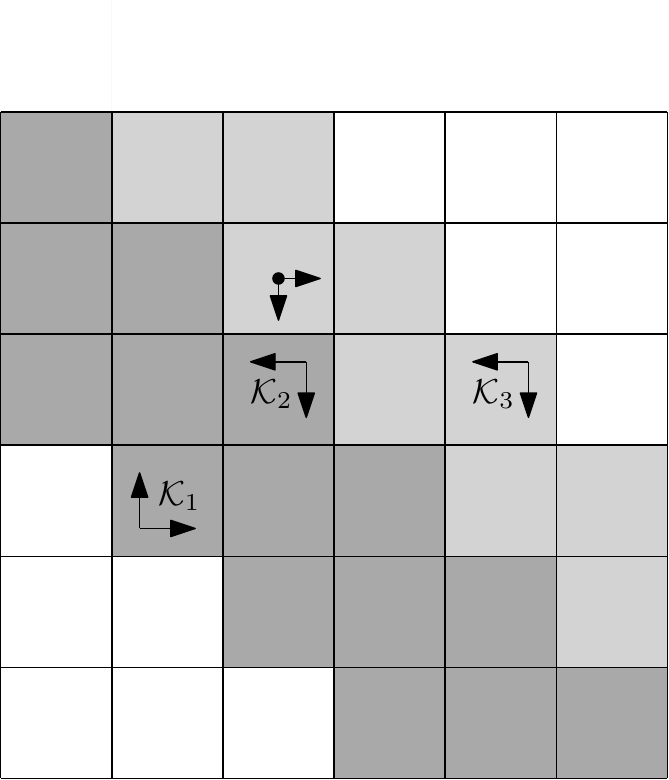}
\caption{Illustration of the situation treated in Lemma \ref{land2}. The robber (black circle) is in $\B(\K_1,\K_3)\setminus\B(\K_1,\K_2)$. The dark squares form $\B(\K_1,\K_2)$, all shaded squares form $\B(\K_1,\K_3)$.} \label{fig:new2}
\end{center}
\end{figure}

\begin{proof}
The situation in this lemma is shown in Figure \ref{fig:new2}, in which the robber is represented by the black circle. (Note that $\K_1$ and $\K_2$ must ``point towards each other" since otherwise $\B(\K_1,\K_3)\setminus\B(\K_1,\K_2)$ would be empty.) We may assume that the robber is at a vertex of type $d$ or $-d$. Note that $d$ is orthogonal to $\tau(\K_1)$. Thus, by forcing the robber to move, he will eventually enter a neighboring conflux $\K$, whose type is parallel to $\tau(\K_1)$. If $\K \in \MD(\K_3)$, the robber is forced to move out of $\B(\K_1,\K_3)$ after exiting the conflux. Otherwise, $\K$ is in $\B(\K_1,\K_2)$. In this case, Lemma \ref{land0} guarantees the capture of a main shadow of the robber.
\end{proof}

\begin{lemma}\label{land3}
Let $\K_1, \K_2$ be such that $\dd(\K_1,\K_2) = \frac{n}{k}-2$. If the robber is not in $\B(\K_1,\K_2)$, then we can capture his main shadow.
\end{lemma}

\begin{proof}
This follows directly from the fact that by forcing the robber to move, he must enter a conflux whose type is parallel to $\tau(\K_1)$, all of which are contained in $\B(\K_1,\K_2)$.
\end{proof}

All the previous results in this section either assume that we already captured a diagonal shadow of the robber or include some assumption about the current position of the robber in their statements. The following is the first result that makes no such assumptions.

\begin{lemma}\label{six}
Seven cops can capture a diagonal shadow of the robber in $G$.
\end{lemma}

\begin{proof}

First, one of the cops will be chasing the robber in order to force him to move, so we only need to show that six cops can capture the robber if he is forced to move. 
Let $\K_1$ and $\K_2$ be confluxes such that $\tau(\K_2)=-\tau(\K_1)$ and $\dd(\K_1,\K_2)=2$. We move 4 cops to their main corners. By Lemma \ref{land1} we can make sure that the robber is in $\B(\K_1,\K_2)^c$ (i.e. out of $\B(\K_1,\K_2)$). By applying Lemma \ref{land0} we can see that, once the robber is in $\B(\K_1,\K_2)^c$, he cannot enter $\B(\K_1,\K_2)$ without having his shadow captured. We will show that we can either capture a diagonal shadow of the robber or force the conditions of Lemma \ref{land3}. 

Let $\K_3$ be a conflux such that $\dd(\K_1,\K_3) = \dd(\K_1,\K_2) + 2$ and cover its main corners with two cops. Notice that Lemma \ref{land2} guarantees that either we capture a diagonal shadow of the robber (in which case we are done) or that the robber is in $\B(\K_1,\K_3)^c$. In the latter case, the cops in $\K_2$ can be released and we can rename $\K_3$ as $\K_2$. This can be done until the diagonal shadow of the robber is caught or $\dd(\K_1,\K_2) = \frac{n}{k}-2$. In the later situation, since the robber is $\B(\K_1,\K_2)^c$, by applying Lemma \ref{land3} we can capture the robber's diagonal shadow.
\end{proof}

The basic idea of the proof of Theorem \ref{k-regular} is to successively capture diagonal shadows of the robber such that their mirrors get closer until the distance between them is two, and then use the remaining cops to capture the robber between those mirrors. However, it is clear that in order to effectively restrict the robber's movements we need two different mirrors. However, there is no way to guarantee that if we have a cop in a diagonal shadow of the robber and we use Lemma \ref{six} again we won't capture the shadow where we already have a cop. A simple way around this problem is to use Lemma \ref{six} twice at the same time. This is what the following result deals with.

\begin{lemma}\label{twelve}

Thirteen cops can capture two main diagonal shadows of the robber simultaneously. Moreover, we can actually guarantee that the distance between the mirrors of the diagonal shadows is two.
\end{lemma}

\begin{proof}
Like in the proof of Lemma \ref{six}, a cop will force the robber to move, so we only need to show that twelve cops can achieve the desired result if the robber is forced to move at such step. For each cop $C$ used in the strategy of Lemma \ref{six}, we will use one more cop $C'$ in the following way: If $p(C) = v$, we will choose $p(C') = v'$, where $v ' \in (\SD(v) \cap \MD(v)) \setminus \{v\}$. Notice that $\tau(v) = \tau(v')$, so  we can move $C'$ in such a way that he stays in the shadow of $C$. In this way, by using the strategy of Lemma \ref{six} with the first set of six cops and maintaining the copies of the cops in their shadows, we will capture two main shadows of the robber simultaneously. Let $\ell$ and $\ell'$ be the mirrors of these shadows, and $C$ and $C'$ the cops moving in these shadows. 

If $\A(\ell,\ell') \leq 2$ we are done, so we can assume it is greater than two. Suppose the type of the confluxes that we used in the application of Lemma \ref{six} is $d$. Notice that whenever $\A(\ell,\ell') \geq 4$, there exist maximal confluxes $\K_1, \K_2$ and $\K_3$ whose types are parallel to $\ell$ such that $\MD(\K_1), \MD(\K_2)$ and $\MD(\K_3)$ are mutually disjoint. 

Since we are using one cop to guard each mirror, we have ten free cops. By using six of those ten cops to repeat the strategy moving positioning the cops in confluxes in $\MD(\K_1)\cup\MD(\K_2)\cup\MD(\K_3)$ of type $d$ we will capture a new diagonal shadow. If $\ell''$ is the mirror corresponding to this new shadow and $C''$ is the cop guarding it, notice that $\max\{\A(\ell,\ell''), \A(\ell',\ell'') \} < \A(\ell,\ell')$ and the robber is either between $\ell''$ and $\ell$ or between $\ell''$ and $\ell'$, so we can release either $C$ or $C'$. Since the robber's movements are restricted to a strictly smaller set, induction over the distance between the mirrors gives us that the distance between mirrors is two. 
\end{proof}

With this we are ready to prove the main theorem of this section:

\begin{thm}
\label{k-regular}
For every $k \geq 2$, if $G$ is a $k$-regularly oriented grid, then $c(G) \leq 13$.
\end{thm}

\begin{proof}
Since we have $13$ cops, we can use one to chase the robber and force him to move. That means we have twelve free cops. By Lemma \ref{twelve}, we can assume 10 of those cops are free and that the robber is restricted to the vertices between two mirrors at distance two. If we manage to capture a main diagonal shadow of the robber between the mirrors whose type is parallel to the mirrors, then we capture the robber.

Let $\K_1, \K_2, \K_3$ and $\K_4$ be confluxes such that $\dd(\K_1,\K_2) = 2$,  $d = \tau(\K_1) - \tau(\K_2)$, and $\K_3 = \K_1 + 2d$ and $\K_4 = \K_2 + 2d$ and guard the main diagonals of each of these four confluxes with two cops. We can assume the robber is in a vertex of type orthogonal to $\tau(\K_1)$. Notice that $\dd(\K_2,\K_3) = \frac{n}{k}-2$. An application of Lemma \ref{land3} to $\K_2$ and $\K_3$ guarantees that the robber is in $\B(\K_2,\K_3)$. By Lemma \ref{land1} applied to $\K_1$ and $\K_2$, and to $\K_3$ and $\K_4$, we can guarantee that the robber is in $\left(\B(\K_1,\K_2) \cup \B(\K_2,\K_3)^c \cup\B(\K_3,\K_4)\right)^c = \B(\K_1,\K_4)^c$. 

Let $t = \frac{n}{k} - \frac{1}{2}\dd(\K_1,\K_4)$. The proof will be by induction on $t$. If $t = 1$, Lemma \ref{land3} guarantees the capture of the robber, so we can assume $t \geq 2$. Suppose $t = m$. Since the robber is in $\B(\K_1,\K_4)^c$, we can release the cops in $\K_2$ and $\K_3$ and move them to the main corners of the confluxes $\K_5 = \K_1 + 4d$ and $\K_6 = \K_4 + 4d$. Again, an application of Lemma \ref{land1} with $\K_5$ and $\K_6$ guarantees that the robber is in $\B(\K_5,\K_6)^c$, and using Lemma \ref{land3} with $\K_4$ and $\K_5$ gives that the robber is in $\B(\K_4,\K_5)$. This now gives us that the robber is in $B(\K_1,K_6)^c$, so if we rename $\K_6$ as $\K_4$ we get that $t = m - 1$, so the result follows by induction.
\end{proof}

It is important to mention that the only part of the proof where we use $13$ cops is during the application of Lemma \ref{twelve}. The rest of the proof only uses $11$ cops, so finding a more efficient way of capturing two diagonal shadows simultaneously would improve the bound for the cop number of $G$.

\section{Paddles}\index{paddle}

Now we return to the general case of straight-ahead oriented toridal grid $C_n\,\square\, C_n$.
We begin by establishing more general conditions which guarantee that the robber is \emph{confined to a stream}, i.e., the robber will be forced to stay in the subgraph of the stream as he will be caught if trying to leave it.
As before, we assume that the robber is forced to move.

\begin{lemma} \label{stream trap} Let $S$ be a stream and let $\ell_1$ and $\ell_2$ be the lines which form the boundary of $S$.
Suppose that the robber is in $S$, at distance $d_1$ from $\ell_1$ and at distance $d_2$ from $\ell_2$.
Let $v_1, v_2$ be the closest vertices (using distances in the undirected grid) of $\ell_1, \ell_2$ to $p(R)$, respectively.
Suppose further that there are distinct cops $C_1$ and $C_2$, such that $C_1$ can move to $v_1$ in $m_1 \leq d_1$ moves and $C_2$ can move to $v_2$ in $m_2 \leq d_2$ moves.
Then by using $C_1$ and $C_2$ we can ensure that the robber will be caught or confined to $S$.
\end{lemma}

\proof If $d_1 = 0$ or $d_2 = 0$ then the robber is already caught.
Otherwise, whatever the robber's move, we update $v_1$, $v_2$, $d_1$ and $d_2$ accordingly.
Now for each $i \in \{1, 2\}$, if $C_i$ is at $v_i$ then he remains in place; otherwise, he moves towards $v_i$.
We will show that the conditions of the lemma are maintained.
If the robber moved in the direction of the stream, then $v_1$ and $v_2$ each move in the direction of the stream and $d_1$ and $d_2$ are unchanged.
In this case the robber's move increases $m_1$ and $m_2$ by at most $1$, and the cops' moves immediately decrease $m_1$ and $m_2$ by $1$, to a minimum of $0$; thus $m_1 \leq d_1$ and $m_2 \leq d_2$.
If the robber moved towards $\ell_1$, then $d_1$ decreases by $1$ and $d_2$ increases by $1$.
The cops' moves now decrease $m_1$ and $m_2$ by $1$, to a minimum of 0, and again $m_1 \leq d_1$ and $m_2 \leq d_2$.
The case in which the robber moved towards $\ell_2$ is similar.
\endproof

Given a conflux $\K$, we refer to the secondary corner with no outneighbours in $\K$ as the \emph{terminal corner} of $\K$.
Let $v$ be the main corner of $\K$ with a vertical (respectively, horizontal) edge leaving $\K$ (but no other edge leaving $\K$, unless $\K$ has only one vertex);
then we refer to the vertical (respectively, horizontal) outneighbour of $v$ as the \emph{vertical guard post} (respectively, \emph{horizontal guard post}) of $\K$.
If $\K$ is maximal, then we refer to the vertex outside $\K$ with the same outneighbours as the terminal corner as the \emph{terminal guard post} of $\K$. See Figure \ref{fig:new3}.

\begin{figure}[htb]
\begin{center}
\includegraphics[totalheight=0.2\textheight]{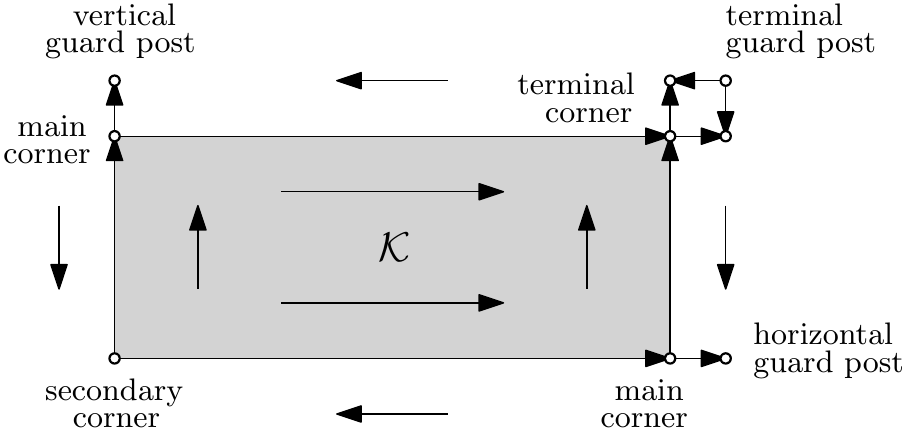}
\caption{A maximal conflux $\K$, its corners and its guard posts.} \label{fig:new3}
\end{center}
\end{figure}

\begin{lemma} \label{trap3} Let $S_1$ be a vertical stream, $S_2$ be a horizontal stream and let $\K$ be the conflux $S_1 \cap S_2$.
Suppose that the robber is in $\K$, for each $i \in \{1, 2\}$, let $d_i$ and $d'_i$ be the distance from $p(R)$ to the boundary of $\K$ in the direction of $S_i$ and in the opposite direction, respectively.
Suppose that there are distinct cops $C_V$, $C_H$ and $C_T$ such that $C_V$ can reach the vertical guard post of $\K$ in $m_V \leq d_1 + d'_2 + 1$ moves, $C_H$ can reach the horizontal guard post of $\K$ in at most $m_H \leq d_2 + d'_1 + 1$ moves, and $C_T$ can reach either the terminal corner or, if $\K$ is maximal, the terminal guard post of $\K$ in $m_T \leq d_1 + d_2$ moves.
Then by committing $C_V$, $C_H$ and $C_T$ we can ensure that the robber will be caught or confined to either $S_1$ or $S_2$.
\end{lemma}

\proof If the robber does not leave $\K$ on his move, then with the cops' moves we will decrease $m_V$, $m_H$ and $m_T$ by $1$, to a minimum of zero; then it is clear that the conditions of the lemma still hold.
Since the robber can make only finitely many such moves, we may assume that the robber leaves $\K$ on his move.
Suppose without loss of generality that he leaves $S_2$ and remains in $S_1$.
In this case, before the robber's move we must have had $d_1 = 0$, and hence $m_V \leq d'_2 + 1$ and $m_T \leq d_2$.
Let $\ell_1$ and $\ell_2$ be the boundary lines of $S_1$, such that there is a directed path in $\K$ from $\ell_2$ to $\ell_1$.
Let $v_1$ and $v_2$ be the closest vertices of $\ell_1$ and $\ell_2$, respectively, to $p(R)$.
Now observe that $C_T$ can move to $v_1$ in at most $d_2 + 1$ moves (by first moving to the terminal corner or terminal guard post of $\K$), while $C_V$ can move to $v_2$ in at most $d'_2 + 1$ moves (since $v_2$ is the vertical guard post of $\K$).
We move each cop one step along the appropriate directed path;
now Lemma~\ref{stream trap} implies that we can ensure the robber will be caught or confined to $S_1$.
\endproof

Our strategy to catch the robber will be based on blocking streams of maximum width and thus confining the robber to be between two vertical or horizontal streams. The basic tool is the following. Let us denote by $w$ the largest width of a stream in $G$. To avoid technicalities in the forthcoming proofs, we will assume that $n\gg w$. The use of this assumption is justified by the covering Lemma \ref{lem:cover}.

\begin{lemma} \label{block brook} Let $S$ be a stream of maximum width, and let $m = \lfloor w(S)/3 \rfloor + 1$.
Let $S'\subseteq S$ be formed by the lines of $S$ that are at distance at least $m - 1$ from the boundary lines of $S$.
Then by committing $60$ cops, and temporarily using additional $60$ cops, we can ensure that after a finite number of steps, if the robber enters $S'$ then he will be caught or confined to a stream.
\end{lemma}

\proof Without loss of generality we may assume that $S$ is a vertical stream, and that the edges of its lines are all directed upwards.
Let $\ell_1$ and $\ell_2$ be the boundary lines of $S$, and let $\ell'_1$ and $\ell'_2$ be the lines outside $S$ adjacent to $\ell_1$ and $\ell_2$, respectively.
We use two formations of cops, which we call paddles: an \emph{inner paddle} (respectively, \emph{outer paddle}) is a formation of 32 cops evenly spaced at distance $m$ along each of $\ell_1$ and $\ell_2$ (respectively, $\ell'_1$ and $\ell'_2$), where each horizontal line has either two or zero cops in each paddle.
By our assumption that $n\gg w$, every paddle has a row containing two cops such that $w$ rows below it contain no cops of the paddle. Let us consider the union of this row together with $12m$ rows above it that contain $26$ cops of our paddle. The union of these rows is the \emph{domain of the paddle}.

\begin{claim} \label{domain trap} If the robber enters the intersection of the domain of a paddle with $S'$ along a horizontal edge, then he will be caught or confined to a stream.
\end{claim}

To prove the claim, we first observe that there is a pair of cops at the same row or below $p(R)$ at a vertical distance of at most $m-1$. If the paddle is an inner paddle, Lemma~\ref{stream trap} implies that the robber will be caught or confined to $S$.
Hence we may assume that the paddle is an outer paddle.
If $m = 1$, then $w(S)\leq 2$ and $S'=S$. Then the robber must have been at the same vertex as a cop on the previous move, which is a contradiction;
hence, $m \geq 2$.
Let $\K$ be the maximal conflux containing the robber; then $\K$ has height at most $w(S)$.
Suppose without loss of generality that the horizontal edges in $\K$ are directed from $\ell_2$ to $\ell_1$.
Then the robber is at distance $m - 1$ from $\ell_2$ and at distance $w(S) - m$ from $\ell_1$.
There is a cop $C_T$ on $\ell'_1$ above the top row of $\K$, at a distance of at most $m$;
this cop can reach the terminal guard post of $\K$ in at most $m-1$ moves.
Further, there is a pair of cops above or level with $p(R)$ at a vertical distance at most $m - 1$.
Let $C_V$ be the cop in this pair which is on $\ell'_2$, and let $C$ be the cop in this pair which is on $\ell'_1$.
We let $C_H = C$ if $C_T \neq C$; otherwise, we let $C_H$ be the closest cop above $C$ on $\ell'_1$.
Let $d_1$ and $d'_1$ be the distances from $p(R)$ to the top and bottom rows of $\K$, respectively.
Then $C_V$ can reach the vertical guard post of $\K$ in at most $(m-1) + 1 + d_1 + 1 = (m - 1) + d_1 + 2$ moves, while $C_H$ can reach the horizontal guard post of $\K$ in at most $(2m - 1) + d'_1 \leq (w(S) - m) + d'_1 + 2$ moves, where the inequality follows from the definition of $m$.
Now $C_V$, $C_H$ and $C_T$ each make their first moves along their respective paths, and the claim follows by Lemma~\ref{trap3}.
Note that $C_T$, $C_V$ and $C_H$ are cops in the paddle that might be above the domain of the paddle. This is the reason that we use 30 instead of 26 cops in the paddles.

\begin{claim} \label{paddle turn} The cops forming an inner paddle can reform to form an outer paddle with the same domain in at most $2w(S) + 1$ moves, and vice versa.
\end{claim}

To prove the claim, we observe that for any vertex of $\ell_1$ there is a directed path of length at most $2w(S) + 1$ to the horizontally adjacent vertex of $\ell'_1$: move up $d \leq w(S)$ times until there is an edge to $\ell'_1$, move to $\ell'_1$, and then move down $d$ times.
Similarly there is a directed path of length at most $2w(S) + 1$ from any vertex of $\ell'_1$, $\ell_2$ or $\ell'_2$ to the horizontally adjacent vertex of $\ell_1$, $\ell'_2$ or $\ell_2$ respectively;
the claim now follows immediately.

When we will use Claim \ref{paddle turn} and have a paddle change from inner to outer, or vice versa, we will term this as \emph{reforming of the paddle}. Note that the domain of the paddle does not change during the reforming process.
Now suppose we have formed the cops into two paddles, each consisting of $30$ cops.
Observe that the domains $P_1$ and $P_2$ of these paddles each contain $12m + 1 \geq 4w(S) + 2$ rows.
Since having a larger domain only helps us, we may assume that $P_1$ and $P_2$ contain $4w(S) + 2$ rows each.

We first show that once we have set up the appropriate circumstances, we can make sure that the robber remains within either $P_1$ or $P_2$ indefinitely.
To achieve this we define five states, and show that regardless of the robber's move we can either stay in the same state or go to the next state (where State 5 is followed by State 1).
We say that a domain is \emph{moving up} (respectively \emph{moving down}) if the corresponding paddle is an inner (respectively outer) paddle.
We say that a domain has \emph{$t$ steps to start} moving up or down if it is in the process of reforming and will complete this process in $t$ moves, and that it is \emph{active} otherwise.
We say that we \emph{switch the domain} when we reform the corresponding paddle to move in the opposite direction.
Note that the States 1-5 assume that our domain is moving up. The classification of states will be equally true if we reverse the vertical directions or relabel the paddles, so that we may do that at any time.

\medskip
\noindent \textbf{State 1:} $P_1$ is moving up, $P_2$ is moving down, $P_1$ and $P_2$ occupy the same rows and the robber is on one of these rows.

In this state the robber can only force us out of State 1 by leaving the occupied rows.
If he does so and moves above the top row of $P_1$,
then we move $P_1$ up, switch $P_2$ and enter State 2. If he leaves the rows of $P_1$ and $P_2$ below, then we move $P_2$ down, switch $P_1$ and enter State 2 (with the role of $P_1$ and $P_2$ exchanged).
Otherwise we remain in State 1.

\noindent \textbf{State 2:} $P_1$ is moving up, $P_2$ is reforming and has $t$ steps to start moving up, $P_1$ is $d_1$ rows above $P_2$ and the robber is $d_2$ rows below the top row of $P_1$, where $d_1 + d_2 + t \leq 2w(S)+1$.

When we come from State 1 to State 2 we have $d_1=1$, $d_2=0$ and $t\leq 2w(S)$. Later in this state if the robber moves above the top row of $P_1$ then we move $P_1$ up; then $d_1$ increases by $1$ and $d_2 = 0$.
Otherwise we keep $P_1$ stationary; then $d_2$ increases by at most $1$ and $d_1$ remains the same.
Hence $d_1 + d_2$ increases by at most $1$; since $t$ decreases by $1$, the inequality still holds and we stay in State 2 until $t = 0$, when we enter State 3.

\noindent \textbf{State 3:} $P_1$ and $P_2$ are both moving up, $P_1$ is $d_1 \geq 1$ rows above $P_2$ and the robber is $d_2$ rows below the top row of $P_1$, where $d_1 + d_2 \leq 2w(S) + 1$.

In this state if the robber moves above the top row of $P_1$ then we move both $P_1$ and $P_2$ up; then $d_1$ remains the same and $d_2 = 0$.
Otherwise we move only $P_2$ up; then $d_1$ decreases by $1$ while $d_2$ increases by at most $1$.
So the inequality still holds and we stay in State 3 unless $d_1 = 0$, when we enter State 4.

\noindent \textbf{State 4:} $P_1$ and $P_2$ are both moving up and occupy the same rows, and the robber is $d \leq 2w(S) + 1$ rows below the top row of $P_1$.

If $d = 2w(S) + 1$ and the robber moves down then we switch $P_2$ and enter State 5.
Otherwise we move $P_1$ and $P_2$ only if the robber goes above their top row, and remain in State 4.

\noindent \textbf{State 5:} $P_1$ is moving up, $P_2$ is reforming and has $t$ steps to start moving down, $P_1$ and $P_2$ occupy the same rows and the robber is $d$ rows below the top row of $P_1$, where $d + t \leq 4w(S) + 2$ and $d - t \geq 0$.

In this case we keep both $P_1$ and $P_2$ stationary and stay in State 5 until $t = 0$, when we enter State 1.

\medskip
We next show that we can reach one of these states.
We form our cops into four paddles of $30$ cops each, with domains $P_1$, $P_2$, $P'_1$ and $P'_2$.
Initially all of these domains occupy the same rows. Then we begin with $P_1$ and $P_2$ moving up and $P'_1$ and $P'_2$ moving down.
At some step, the robber will occupy either the top row of $P_1$ and $P_2$ or the bottom row of $P'_1$ and $P'_2$.
Without loss of generality he occupies the top row of $P_1$ and $P_2$.
At this point we enter State 3, and release the cops from $P'_1$ and $P'_2$.

Now if at any point the robber is outside $S'$, Claim~\ref{domain trap} implies that he cannot re-enter $S'$ without being caught or confined to a stream.
To force the robber to leave $S'$, we choose an arbitrary row  and place two cops at either end of the intersection of this row with $S$.
One cop on each end remains stationary, while the remaining two cops move in the direction of $S$.
If the robber does not leave $S'$, then after some finite time he will be on the same row as one of the pairs of cops, at which point he is confined to a stream.
\endproof

In the sequel we will use Lemma \ref{block brook} on streams that may not be maximum width. The only feature where we needed to use that $S$ has large width was that any crossing stream in the domain of the paddle is not wider than $w(S)$. In order to have this property, some confluxes of the considered stream will be guarded (by involving 3 cops for each such conflux by using Lemma \ref{trap3}).

\begin{thm} \label{thm:any directed cycles} 
If $G$ is any straight-ahead orientation of a toroidal grid, then $c(G) \leq 319$.
\end{thm}

\proof 
For each stream $S$ mentioned below, $S'$ is formed by the lines of $S$ at distance at least $m-1$ from the boundary lines of $S$
(or all the lines of $S$, if $m \leq 1$), where $m = \lfloor w(S)/3 \rfloor+1$.
For any subgraph $H$ of $G$, we define $S(H)$ to be a stream of maximum width intersecting $H$.

Throughout the proof, let $T$ be the \emph{territory of the robber}, i.e., vertices of the graph that the robber can reach without getting caught. Also let $H$ be $G-U$ where $U$ is the union of guarded streams.

Let $S_1, S_2$ and $S_3$ be three widest streams in $G$. We will have two cases.

\textbf{Case 1.} $S_1, S_2$ and $S_3$ are (without loss of generality) vertical streams. Then Lemma \ref{block brook} shows that we can commit $3\times 60$ cops (and temporarily using another 60 cops) to guard $S_1, S_2$ and $S_3$. The robber is confined between two of these streams, say $S_1$ and $S_2$.
Now we can release the cops used to guard $S_3$.  Now, redefine $S_3$ to be equal to $S(T\cap H)$ (i.e., the widest stream intersecting the territory of the robber).
We can continue this approach to shrink the territory of the robber until $S_3$ is a horizontal stream, which brings us to Case 2.

\textbf{Case 2.} $S_1$ and $S_2$ are vertical streams but $S_3$ is a horizontal one. Then we can commit $2\times 60$ cops to guard $S_1$ and $S_2$. Let $\K_1$ and $\K_2$ be the intersection of $S_1$ and $S_2$ with $S_3$, respectively. Using Lemma \ref{trap3}, we can commit $2\times 3$ cops to guard $\K_1$ and $\K_2$ and then using 60 cops we can guard $S_3$. Note that there is no stream wider than $S_3$ that intersects $S_3$ in $T\cap H$. Also note that since the horizontal movement of the robber is bounded between $S_1$ and $S_2$, we do not need (even temporarily) an extra 60 cops to guard $S_3$.
Now, let $S_4$ be the next widest stream intersecting $T\cap H$. 

\textbf{Subcase 2a.} If $S_4$ is a vertical stream, then use another 3 cops to guard $\K_3:=S_3\cap S_4$ and use 60 cops to guard $S_4$. Now, based on the position of the robber, redefine $S_1$ and $S_2$, release the third set of 60 cops, shrink the territory of the robber and repeat Subcase 2a.

\textbf{Subcase 2b.} If $S_4$ is a horizontal stream, then commit $2\times 3$ cops to protect $\K_3:=S_1\cap S_4$ and $\K_4:=S_2\cap S_4$ and use 60 cops to guard $S_4$. Now the robber is confined between $S_1, S_2, S_3$ and $S_4$.

Let $S_5$ be the next widest stream intersecting $T\cap H$. Without loss of generality we can assume that it is a vertical one. Commit $2\times 3$ cops to protect $\K_5:=S_5\cap S_3$ and $\K_6:=S_5\cap S_4$ and use 60 cops to guard $S_5$. Now, based on the position of the robber, redefine $S_1, S_2, S_3$ and $S_4$, release the fifth set of 60 cops, shrink the territory of the robber and repeat Subcase 2b until we catch the robber.

Note that we have used at most 5 sets of 60 cops to guard $S_i$ ($i=1,\ldots, 5$) and at most $6\times 3$ cops to guard the intersections of these streams. Also, to avoid complication, we use one cop to force the robber to move. 
Therefore by using at most 319 cops we can capture the robber.
\endproof


\bibliographystyle{plain}
\bibliography{reference}

\end{document}